\documentclass{amsart}
\usepackage[english]{babel}
\usepackage{amsthm}
\usepackage{inputenc}
\usepackage{amssymb}
\usepackage{graphicx}
\newtheorem{thm}{Theorem}[section]

\newtheorem{lem}[thm]{Lemma}
\newtheorem{prop}[thm]{Proposition}
\newtheorem{defn}[thm]{Definition}
\newtheorem{exam}[thm]{Example}
\newtheorem{rem}[thm]{Remark}
\numberwithin{equation}{section}

\def\ll{\mathcal{L}}
\begin{document}

\title[The classification of 4-dimensional Leibniz algebras]{The classification of 4-dimensional Leibniz algebras}%
\author{Elisa M. Ca\~{n}ete,  Abror Kh. Khudoyberdiyev}

%

\begin{abstract}
This paper is a contribution to the development of the non associative algebras theo\-ry. More precisely, this work deals with the classification of the complex $4$-dimensional Leibniz algebras. Note that the classification of $4$-dimensional nilpotent complex Leibniz algebras was obtained
in \cite{3-dim}. Therefore we will only consider non nilpotent case in this work.
\end{abstract}
\maketitle
\section{Introduction.}
Non associative algebras appear at the beginning of the twentieth century as a consequence of the development of quantum mechanics.
Pascual Jordan, John von Neumann and Eugene Wigner were the first researchers in introducing this kinds of algebras - in particular,
Jordan algebras - in 1934 and then, Jean-Louis Loday introduced the Leibniz algebras in 1993 in his cyclic homology study \cite{loday}. Loday explained in his study that the main property of the Lie brackets in the homologic study was actually the Leibniz property. In this way, the Leibniz algebras are the natural generalization of the Lie one. Moreover, Leibniz algebras inherit an important property of Lie algebras which is that the right multiplication operator of a Leibniz algebra is a derivation.

Active investigations on Leibniz algebras theory show that many results of the theory of Lie algebras can be extended to Leibniz one. Of course, distinctive properties of non-Lie Leibniz algebras have been also studied \cite{AAO1,AyOm1,AyOm2,Bar}.

 Thanks to Levi-Malcev's theorem, the solvable Lie algebras have played an important role in ma\-the\-matics over the last years, either in the classification theory or in geometrical and analytical appli\-cations. The investigation of this kind of algebras with some special types of nilradicals comes from different problems in Physics and was the subject of various paper \cite{Ancochea,Boyko,Campoamor,Muba,Tremblay,Wang}. In Leibniz algebras, the analogue of Levi-Malcev's theorem was proved in recent times in \cite{Bar}, thus solvable Leibniz algebras still play a central role and have been recently studied by several authors, like Casas, Ladra, Karimjanov and Omirov. In their recent works, some interesting results have been obtained. In particular, they have tackled the classification of $n$-dimensional solvable Leibniz algebras with some restriction in the nilradical (see \cite{Casas, Casas2}).

 The classification, up to isomorphism, of any class of algebras is a fundamental and very difficult problem. It is one of the first problem that one
encounters when trying to understand the structure of a member of this class of algebras.

The purpose of the present work is to continue the study of Leibniz algebras. Since the description of the whole $n$-dimensional family seems to be unsolvable, we reduce our discussion to the restriction on the dimension of it. Therefore we reduce our attention to the classification of $4$-dimensional Leibniz algebras. According to the analogue to Levi-Malcev's theorem, our aim lies in the solvable family because the description of simple
Leibniz algebras immediately follows. More exactly, Rakimov, Omirov and Turdibaev have proved in \cite{Rak} that there exists only one non
solvable $4$-dimensional Leibniz algebra $sl_2\oplus
\mathbb{C}=\{e,f,h,x\}.$
$$sl_2\oplus
\mathbb{C}: \ [e,h]=-[h,e]=2e, \quad [h,f]=-[f,h]=2f, \quad
[e,f]=-[f,e]=h.$$

 Moreover, Albeverio, Omirov and Rakhimov's work \cite{3-dim} has been fundamental in this research, where the classification of $4$-dimensional nilpotent complex Leibniz algebras was obtained. Therefore we focus our attention in the study of the $4$-dimensional non-nilpotent solvable Leibniz algebras.

To the description of $4$-dimensional solvable Leibniz algebras we use the method for describing solvable Lie algebras with given nilradical by
 means of non-nilpotent outer derivations of the nilradical.

Note that this method has been used for Lie algebra in several papers\cite{AnCaGa,NdWi,SnWi}, and for Leibniz ones has been used, for instance, in \cite{Casas, Casas2}.

\section{Preliminaries.}

Let us introduce some definitions and notations, all of them
necessary for the understanding of this work.

\begin{defn}
A Leibniz algebra over $K$ is a vector space $\mathcal{L}$ equipped with a bilinear map, called bracket,
$$[-,-]:  \mathcal{L} \times  \mathcal{L} \rightarrow  \mathcal{L} $$
satisfying the Leibniz identity:
$$[x,[y,z]]=[[x,y],z]-[[x,z],y],$$
for all $x,y,z \in  \mathcal{L}.$
\end{defn}
\begin{exam}
Any Lie algebra is a Leibniz algebra.
\end{exam}

From now on Leibniz algebras $\ll$ will be considered over the
field of complex numbers $\mathbb{C}$, and with finite dimension.

The set $R(\ll)=\{x \in \ll: [y,x]=0, \ \forall y \in \ll\}$ is
called \emph{the right annihilator of $\ll$}. Notice that for any
$x, y \in \ll$ the elements $[x,x]$ and $[x,y]+ [y,x]$ are always
in $R(\ll)$ and that $R(\ll)$ is always an ideal of $\ll$.

For a given Leibniz algebra $(\ll,[-,-])$ the sequences of
two-sided ideals defined recursively as follows:
\[\ll^1=\ll, \ \ll^{k+1}=[\ll^k,\ll],  \ k \geq 1, \qquad \qquad
\ll^{[1]}=\ll, \ \ll^{[s+1]}=[\ll^{[s]},\ll^{[s]}], \ s \geq 1.
\]
are said to be the lower central and the derived series of $\ll$,
respectively.

\begin{defn} A Leibniz algebra $L$ is said to be
nilpotent (respectively, solvable), if there exists $n\in\mathbb
N$ ($m\in\mathbb N$) such that $\ll^{n}=0$ (respectively,
$\ll^{[m]}=0$). The minimal number $n$ (respectively, $m$) with
such property is said to be the index of nilpotency (respectively,
of solvability) of the algebra $\ll$.
\end{defn}

Evidently, the index of nilpotency of an $n$-dimensional algebra is not greater than $n+1$.

 \begin{defn}
  The maximal nilpotent ideal of a Leibniz algebra is said to be the nilradical of the algebra.
  \end{defn}

\begin{defn} A linear map $d \colon \ll \rightarrow \ll$ of a Leibniz algebra $(\ll,[-,-])$ is said to be a
 derivation if for all $x, y \in \ll$, the following  condition holds: \[d([x,y])=[d(x),y] + [x, d(y)] \,.\]
\end{defn}

Let $\ll$ be a Leibniz algebra, the set $R(\ll)=\{x \in \ll: [y,x]=0, \ \forall y \in \ll\}$ is called \emph{the right annihilator of $\ll$}. Notice that for any $x, y \in \ll,$ the elements $[x,x]$ and $[x,y]+ [y,x]$ are always in $R(\ll)$ and that $R(\ll)$ is always an ideal of $\ll$.

For a given $x \in \ll,$ $R_x$ denotes the map $R_x: \ll \rightarrow \ll$ such that $R_x(y)=[y,x], \ \forall x \in \ll$.
Note that the map $R_x$ is a derivation. We call this kind of derivations as inner derivations. Derivations that are not inner are said to be outer derivations.

\begin{defn}\cite{Muba}
Let $d_1, d_2, \dots, d_n$ be derivations of a Leibniz algebra
$\ll.$ The derivations $d_1, d_2, \dots, d_n$ are said to be
nil-independent if $$\alpha_1d_1 + \alpha_2d_2+ \dots +\alpha_nd_n$$
is not nilpotent for any scalars $\alpha_1, \alpha_2, \dots
,\alpha_n \in \mathbb{F}.$ In other words, if for any $\alpha_1, \alpha_2, \dots ,\alpha_n \in \mathbb{F}$
there exists a natural number $k$ such that $$(\alpha_1d_1 + \alpha_2d_2+ \dots +\alpha_nd_n)^k=0 \hbox{ then } \alpha_1 = \alpha_2= \dots =\alpha_n.$$
\end{defn}

 Section $3$ of this paper is divided into two subsections.
 In the first one we study the $4$-dimensional solvable Leibniz algebras with
 $2$-dimensional nilradical, while in the second subsection we show an analogous study when the nilradical has dimension equals $3$.


\section{$4$-dimensional solvable Leibniz algebras}

Let $\ll$ be a solvable Leibniz algebra. Then it can be decomposed
in the form $\ll=N + Q,$ where $N$ is the nilradical and $Q$ is
the complementary vector space. Since the square of a solvable
algebra is a nilpotent ideal and the finite sum of nilpotent
ideals is a nilpotent too \cite{AyOm1}, then we get the
nilpotency of the ideal $\ll^2,$ i.e, $\ll^2 \subseteq N$ and
consequently, $Q^2 \subseteq N.$ Casas, Ladra, Omirov and
Karimjanov proved in \cite{Casas} the following theorem, that will
be very useful in this work.
\begin{thm}\label{nilr}
Let $\ll$ be a solvable Leibniz algebra and $N$ its nilradical. Then the dimension of the complementary vector space to $N$ is not greater than the maximal number of nil-indepedent derivations of N.
\end{thm}

According to Theorem \ref{nilr}, we assert that the dimension of the nilradical of $4-$dimensional solvable Leibniz algebras
are equal to two or three. The classification of the two and three dimensional nilpotent Leibniz algebras
was obtained in \cite{Loday} and \cite{AAO1}, respectively. The following theorems show these classifications.

\begin{thm}\label{2dim-nil}\cite{Loday}
Let $\ll$ be a $2$-dimensional nilpotent Leibniz algebra. Then $\ll$ is an abelian algebra or it is isomorphic to
$$\mu_1: [e_1,e_1]=e_2$$
\end{thm}

\begin{thm} \label{lambda}\cite{AAO1}
Let $\ll$ be a $3$-dimensional nilpotent Leibniz algebra. Then $\ll$ is isomorphic to one of the following pairwise non-isomorphic algebras:
\begin{align*}
\lambda_1&: \ abelian,\\
\lambda_2&: \  [e_1,e_1]=e_3,\\
\lambda_3&: \ [e_1,e_2]=e_3, \ [e_2,e_1]=-e_3,\\
\lambda_4(\alpha)&: \ [e_1,e_1]=e_3, \ [e_2,e_2]=\alpha e_3, \ [e_1,e_2]=e_3,\\
\lambda_5&: \ [e_2,e_1]=e_3,\ [e_1,e_2]=e_3,\\
\lambda_6&: \ [e_1,e_1]=e_2, \ [e_2,e_1]=e_3.
\end{align*}
\end{thm}

\begin{rem}
In \cite{3-dim} the law of $\lambda_5$ is defined by the following products:
$$[e_1,e_1]=e_3,\quad [e_2,e_1]=e_3,\quad [e_1,e_2]=e_3.$$

It is enough to take the basis transformation
$e_1'=e_1-\frac{1}{2}e_2,$ to obtain the expression of
$\lambda_5$ in the Theorem \ref{lambda}.
\end{rem}

In order to simplify the below calculations, it is worthwhile to consider the following Lemma.
\begin{lem}
The algebra $\lambda_4(\alpha)$ is isomorphic to one of the following algebras:
$$\begin{array}{cc}
\lambda_4'(\beta):\begin{cases}
[e_2,e_1]=e_3,\\
[e_1,e_2]=\beta e_3, \hbox{ with } \beta=
\frac{\sqrt{1-4\alpha}-1}{\sqrt{1-4\alpha}+1}, \alpha \neq \frac 1
4,
\end{cases} &
\lambda_4':\begin{cases}
[e_1,e_1]=e_3,\\
[e_2,e_1]=e_3,\\
[e_1,e_2]=-e_3.
\end{cases}
\end{array}$$
\end{lem}
\begin{proof}
Let us distinguish two cases:
\begin{itemize}
\item{If $\alpha=0,$} just taking the basis transformation $e_2'=e_1-e_2,$ $e_1'=e_2,$ the algebra $\lambda_4(0)$ is obtained.
\item{If $\alpha \neq 0,$} taking the basis transformation $e_2'=e_2+\beta e_1$ with $\beta=- \frac{1+\sqrt{1-4\alpha}}{2},$
we obtain the following family of algebras:
    $$\begin{cases}
    [e_1,e_1]=e_3,\\
    [e_2,e_1]=\beta e_3,\\
    [e_1,e_2]=(1+\beta)e_3.
    \end{cases}$$
    It is clear that $\beta \neq 0.$ Therefore we can take the basis transformation
     $e_2'=\frac{1}{\beta}e_2,$ and the law of $\lambda_4(\beta)$ can be written as follows:
    $$\begin{cases}
    [e_1,e_1]=e_3,\\
    [e_2,e_1]=e_3,\\
    [e_1,e_2]=(1+\frac{1}{\beta})e_3.
    \end{cases}$$
    Finally, it is useful to consider the following cases:
    \begin{itemize}
    \item{Case 1: If $\beta \neq -\frac{1}{2},$} that is, if $\alpha \neq \frac{1}{4},$ by taking the basis transformation
     $e_1'=e_1-\frac{\beta}{2\beta+1}e_2$ and $\beta'=\frac{1+\beta}{\beta},$ the algebra $\lambda_4'(\beta)$ is obtained, for
    $\beta=
\frac{\sqrt{1-4\alpha}-1}{\sqrt{1-4\alpha}+1}, \alpha \notin \{0,\frac 1 4\}.$
    \item{Case 2: If $\beta=-\frac{1}{2},$} the algebra $\lambda_4'$ is obtained.
    \end{itemize}
\end{itemize}
\end{proof}

\begin{prop}\label{derivations} There exist a basis such that the derivations of the algebras $\lambda_1,$
 $\lambda_2,$ $\lambda_3,$ $\lambda_4',$ $\lambda_4(\beta),$
 $\lambda_5$ and $\lambda_6$ have the following forms:
 $$\begin{array}{lll}
 Der(\lambda_1)=\left(\begin{array}{ccc}
 a_1 & a_2 & a_3 \\
 b_1 & b_2 & b_3 \\
 c_1 & c_2 & c_3 \\
 \end{array} \right),&
 Der(\lambda_2)=\left(\begin{array}{ccc}
  a_1 & a_2 & a_3 \\
 0 & b_2 & b_3 \\
 0 & 0 & 2a_1 \\
 \end{array}\right), &
 Der(\lambda_3)=\left(\begin{array}{ccc}
  a_1 & a_2 & a_3 \\
 b_1 & b_2 & b_3 \\
 0 & 0 & a_1+b_2 \\
 \end{array}\right),\\[7mm]
  Der(\lambda_4')=\left(\begin{array}{ccc}
 a_1 & a_2 & a_3 \\
 0 & 2a_1 & b_3 \\
 0 & 0 & 3a_1 \\
 \end{array}\right), &
 Der(\lambda_4'(\beta))=\left(\begin{array}{ccc}
 a_1 & 0 & a_3 \\
 0 & b_2 & b_3 \\
 0 & 0 & a_1+b_2 \\
 \end{array}\right), &
 Der(\lambda_5)=\left(\begin{array}{ccc}
a_1 & 0 & a_3 \\
 0 & b_2 & b_3 \\
 0 & 0 & a_1+b_2 \\
 \end{array} \right),\\[7mm]
 Der(\lambda_6)=\left(\begin{array}{ccc}
a_1 & 0 & a_3 \\
 0 & 2a_1 & a_2 \\
 0 & 0 & 3a_1 \\
 \end{array} \right).
 \end{array}$$
 \end{prop}
 \begin{proof}
The proof is carried out by checking the derivation property on algebras $\lambda_i.$
 \end{proof}

\subsection{$4$-dimensional solvable Leibniz algebras with $2$-dimensional nilradical}
 In this subsection we focus our attention in the algebras whose nilradical has dimension two. The following theorem shows the classification of these algebras.
\begin{thm}
Let $\ll$ be a $4$-dimensional solvable Leibniz algebra with $2$-dimensional nilradical.
Then, $\ll$ is isomorphic to one of the following pairwise non-isomorphic algebras:
$$\begin{array}{ccc}
\mathcal{R}_1:\begin{cases}
[e_1,x]=e_1,\\
[e_2,y]=e_2,\\ \end{cases}& \mathcal{R}_2:\begin{cases}
[e_1,x]=e_1,\\
[e_2,y]=e_2,\\
[x,e_1]=-e_1,\\
[y,e_2]=-e_2,\\ \end{cases}& \mathcal{R}_3:\begin{cases}
[e_1,x]=e_1,\\
[e_2,y]=e_2,\\
[y,e_2]=-e_2.
\end{cases}
\end{array}$$
\end{thm}
\begin{proof}
First of all, note that the classification of the $2$-dimensional
nilpotent Leibniz algebras is showed in Theorem \ref{2dim-nil}.
Let $\ll$ be a $4$-dimensional solvable Leibniz algebra with
$2$-dimensional nilradical $\mu$. Since the number of
nil-independent derivations of $\mu_1$ equals 1, then we conclude
that $\mu$ must to be an abelian algebra. Therefore, by
considering the basis $\{x,y,e_1,e_2\},$ we have the following
information about the structural constants of $\ll$:
$$\begin{cases} [e_1,x]=a_1e_1+a_2e_2,&
[e_2,x]=a_3e_1+a_4e_2,\\
[e_1,y]=b_1e_1+b_2e_2,&
[e_2,y]=b_3e_1+b_4e_2.
\end{cases}$$
As $R_x$ and $R_y$ are nil-independent derivations of $\mu,$ then we can consider $a_1 \neq 0.$ Let us assume $a_1 =1,$ without loss of generality, and let us take the basis transformation
$y'=y-b_1x.$ Therefore we conclude that $b_1=0.$

It is easy to prove that the matrix of $R_x$
has one of the following forms:
$$\begin{array}{ccc}
\left( \begin{array}{cc}
1 & 0 \\
0 & a_4
\end{array} \right) &
\hbox{ or } &
\left( \begin{array}{cc}
1 & 1 \\
0 & 1
\end{array} \right)
\end{array}.$$
Therefore, we consider the following cases:

 \textit{\bf{Case 1: If $\left( \begin{array}{cc}
1 & a_2 \\
a_3 & a_4
\end{array} \right) \simeq
\left( \begin{array}{cc}
1 & 0 \\
0& a_4
\end{array} \right),$}} then we can write:
$$\begin{cases}
[e_1,x]=e_1,& [e_2,x]=a_4e_2,\\
[e_1,y]=b_2e_2,& [e_2,y]=b_3e_1+b_4e_2.
\end{cases}$$

 By the Leibniz identities $[e_1,[x,y]]$ and $[e_2,[x,y]],$ we obtain the restrictions $b_2(1-a_4)=0$ and $b_3(1-a_4)=0.$
 Thus, it is useful to consider the following subcases:

 \textit{\bf{Case 1.1: If $a_4=1,$}} let us discuss about the structural constants $b_i,$ for $2 \leq i \leq 4.$
 It is clear that $R_y=\left( \begin{array}{cc}
0 & b_2 \\
b_3 & b_4
\end{array} \right)$ is isomorphic to one of the following matrices:
$$\begin{array}{ccc}
\left( \begin{array}{cc}
1 & 0 \\
0 & b_4
\end{array} \right) &
\hbox{ or } &
\left( \begin{array}{cc}
1 & 1 \\
0 & 1
\end{array} \right)
\end{array}.$$

It is easy to prove that
$R_y \simeq \left( \begin{array}{cc}
1 & 0 \\
0 & b_4
\end{array} \right),$ with $b_4 \neq 1.$ Otherwise $R_x-R_y$ would be nilpotent, which is impossible.

By taking the basis transformation $x'=\frac{b_4}{b_4-1}x-\frac{1}{b_4-1}y$ and $y'=\frac{1}{b_4-1}y-\frac{1}{b_4-1}x,$ we obtain the multiplication
$$[e_1,x]=e_1,\quad [e_2,x]=0, \quad [e_1,y]=0,\quad [e_2,y]=e_2.$$

 \textit{\bf{Case 1.2: If $a_4 \neq 1,$}} then we have $b_2=b_3=0.$
As $R_x$ and $R_y$  are nil-independent derivations of $\mu,$  we obtain $b_4\neq 0.$
By taking the change of basis $x' =x-\frac {a_4}{b_4}y, \    y' = \frac 1{b_4}y,$ we have again
$$
[e_1,x]=e_1,\quad
[e_2,x]=0,\quad
[e_1,y]=0,\quad
[e_2,y]=e_2.
$$

 \textit{\bf{Case 2: If $R_x=
\left( \begin{array}{cc}
1 & 1 \\
0& 1
\end{array} \right),$}} by the Leibniz identity $[e_1,[x,y]],$ we have $b_3=b_4=0,$ giving rise to a contradiction with the assumption of $R_y$ is non-nilpotent.

Therefore, we have proved that:
$$
[e_1,x]=e_1,\quad
[e_2,x]=0,\quad
[e_1,y]=0,\quad
[e_2,y]=e_2.
$$

Let us study the remaining products to determinate the law of the algebra $\ll.$ Let us denote
$$\begin{cases}
[x,e_1]=\alpha_1e_1+\alpha_2e_2,&
[x,e_2]=\alpha_3e_1+\alpha_4e_2\\
[y,e_1]=\beta_1e_1+\beta_2e_2,&
[y,e_2]=\beta_3e_1+\beta_4e_2,\\
[x,x]=c_1e_1+c_2e_2,&
[x,y]=c_3e_1+c_4e_2,\\
[y,x]=d_1e_1+d_2e_2,& [y,y]=d_3e_1+d_4e_2.
\end{cases}$$
By taking the basis transformation $x'=x-c_1e_1-c_4e_2,$  $y'=y-d_1e_1-d_4e_2,$ we have $$c_1=c_4=d_1=d_4=0.$$

By using  again the Leibniz identity we get $$\alpha_2=\alpha_3=\alpha_4=\beta_1=\beta_2=\beta_3=c_2=c_3=d_2=d_3=0,$$
$$\alpha_1^2+\alpha_1=0, \quad \beta_3^2+\beta_3=0.$$

Let us distinguish the following cases:

 \textit{\bf{Case 2.1: If  $\alpha_1=\beta_3=0,$}} the algebra $\mathcal{R}_1$ is obtained.

  \textit{\bf{Case 2.2: If $\alpha_1=\beta_3=-1,$}} the algebra $\mathcal{R}_2$ is obtained.

   \textit{\bf{Case 2.3: If $(\alpha_1,\beta_3)=(0,-1)$ or $(\alpha_1,\beta_3)=(-1,0),$}} the algebra $\mathcal{R}_3$ is obtained.
\end{proof}

 \subsection{$4$-dimensional solvable Leibniz algebras with $3$-dimensional nilradical}

 Note that, as the algebras $\lambda_i,$ for $i \in \{1,2,3,5,6\},$ $\lambda_4'$ and the family $\lambda_4(\beta)$ are $3$-dimensional and the goal of this work is to obtain the classification of the $4$-dimensional solvable Leibniz algebras, we are only interested in algebras with one nil-independent derivation of its nilradical in this subsection.

Let $\ll$ be a $4$-dimensional solvable Leibniz algebra, with
$3$-dimensional nilradical $N,$ then there exists a basis $\{x,
e_1, e_2,e_3\}$ of $\ll$ such that the right multiplication operator $R_x$ is a non-nilpotent derivation on $N=\{e_1,
e_2,e_3\}.$

 \begin{prop} \label{propo1}
 Let $\ll$ be a $4$-dimensional solvable Leibniz algebra, whose nilradical is isomorphic to $\lambda_3.$
Then, $\ll$ is isomorphic to one of the following pairwise non-isomorphic algebras:
$$\small \begin{array}{ccc}
 \ll_1(\gamma):\begin{cases}
 [e_1,e_2]=e_3,\\
 [e_2,e_1]=-e_3, \\
 [e_1,x]=e_1, \\
 [e_2,x]=\gamma e_2,\\
 [e_3,x]=(1+\gamma)e_3,\\
 [x,e_1]=-e_1,\\
 [x,e_2]=-\gamma e_2,\\
 [x,e_3]=-(1+\gamma)e_3.
 \end{cases} &
\ll_2:\begin{cases}
 [e_1,e_2]=e_3,\\
 [e_2,e_1]=-e_3, \\
 [e_1,x]=e_1, \\
 [e_2,x]=-e_2,\\
 [x,e_1]=-e_1,\\
 [x,e_2]=e_2,\\
 [x,x]=e_3.
 \end{cases}
&
 \ll_3:\begin{cases}
 [e_1,e_2]=e_3,\\
 [e_2,e_1]=-e_3, \\
 [e_1,x]=e_1+e_2, \\
 [e_2,x]=e_2,\\
 [e_3,x]=2e_3,\\
 [x,e_1]=-e_1-e_2,\\
 [x,e_2]=-e_2,\\
 [x,e_3]=-2e_3.
 \end{cases}
 \end{array}$$
 \end{prop}

 \begin{proof}
 According to Proposition \ref{derivations} and the law of $\lambda_3,$ we know the following products of the structure of the algebra $\ll$:
 \begin{align*}
 [e_1,e_2]&=e_3,\\
 [e_2,e_1]&=-e_3,\\
 [e_1,x]&=a_1e_1+a_2e_2+a_3e_3,\\
 [e_2,x]&=b_1e_1+b_2e_2+b_3e_3,\\
 [e_3,x]&=(a_1+b_2)e_3.
 \end{align*}
 Moreover, since $e_1,e_2 \notin R(\ll),$ applying the properties of the right annihilator we can write:
 \begin{align*}
 [x,e_1]&=-a_1e_1-a_2e_2+\alpha_3e_3,\\
 [x,e_2]&=-b_1e_1-b_2e_2+\beta_3e_3,\\
 [x,e_3]&=-(a_1+b_2)e_3,\\
 [x,x]&=\gamma_3e_3.
 \end{align*}
 Let us distinguish the following cases:

 \textit{\bf{Case 1.}} If $(a_1,b_2) \neq (0,0),$ we can take, without loss of generality, $a_1\neq 0.$ Note that $b_1=0.$ Otherwise, we could consider two cases: $a_2=0$ or $a_2 \neq 0.$ In the first case, by making the change of basis $e_1'=e_2,\ e_2'=e_1,\ e_3'=-e_3,$ we can assume $b_1=0.$ On the other hand, if $a_2\neq 0,$  by making the basis transformation $e_2'=e_2+\frac{-(b_2-a_1)+\sqrt{(b_2-a_1)^2+4a_2b_1}}{2a_2}e_1,$ we come to $b_1=0.$ It suffices to make the basis transformation $x'=\frac{1}{a_1}x$ to obtain $a_1=1.$

    Let us to consider the following cases:

    \textit{\bf{Case 1.1.}} If $b_2 \notin \{0,1\},$ then by means of the following basis transformation:
     $$
     e_1'=e_1-\frac{a_2}{b_2-1}e_2-\frac{a_3(b_2-1)-a_2b_3}{b_2(b_2-1)}e_3,\quad
     e_2'=e_2-b_3e_3,\\
     $$
      we obtain $a_2=a_3=b_3=0.$

     By the Leibniz identities $[x,[e_1,x]],$ $[x,[e_2,x]]$ and $[x,[x,x]]$ we conclude $$\alpha_3=0,\quad \beta_3=0, \quad \gamma_3(1+b_2)=0.$$

     Therefore we have the following possibilities:
     \begin{itemize}
     \item{If $\gamma_3=0,$} the algebra $\ll_1(b_2)$ is obtained for $b_2\notin \{0,1\}$.
     \item{If $\gamma_3 \neq 0,$} then $b_2=-1.$ It is enough to make the following basis transformation:
     $$
     e_1'=\gamma_3e_1,\quad
     e_3'=\gamma_3 e_3,\quad
     $$
     to obtain the algebra $\ll_2.$
     \end{itemize}

     \textit{\bf{Case 1.2.}} If $b_2\in \{0,1\}$, then $e_3 \notin R(L)$ and thanks to the properties
$[x,x], [e_i,x]+[x,e_i] \in R(\ll)$ for $1\leq i \leq 2,$ we conclude
$\gamma_3=0,$ $\alpha_3=-a_3$ and $\beta_3=-b_3.$
\begin{itemize}
\item If $b_2=0,$ once the following basis transformation is done:
   $$e_1'=e_1+a_2e_2,\quad
     e_2'=e_2-b_3e_3, \quad x'=x-(a_3+a_2b_3)e_2$$
     we assert $a_2=a_3=b_3=0.$
Thus we obtain the algebra $\ll_1(0).$

\item If $b_2=1,$ then making the following change of basis:
$$
e_1'=e_1-(a_3+\alpha_2b_3)e_3,\quad
e_2'=e_2-b_3e_3,\\
$$
         we come to $a_3=b_3=0.$

         Finally if $a_2=0,$ we obtain the algebra $\ll_1(1).$
         Otherwise, if $a_2 \neq 0,$ we assert $a_2=1$ by considering the basis transformation $e_2'=a_2e_2,$  $e_3'=a_2e_3.$ Therefore, the algebra $\ll_3$ is obtained.

\end{itemize}

          \textit{\bf{Case 2.}} If $(a_1,b_2)=(0,0),$ due to the non-nilpotency $R_x,$ we assert $a_2b_1 \neq 0.$
Taking the change $e'_1=e_1+e_2,$ we obtain $$[e'_1, x]=[e_1+e_2,x] = b_1(e_1+e_2)+(a_2-b_1)e_2+(a_3+b_3)e_3=b_1e_1'+a_2'e_2+a_3'e_3.$$
Since $b_1 \neq 0,$ we occur to Case 1.
      \end{proof}

  \begin{prop} \label{4}
 Let $\ll$ be a $4$-dimensional solvable Leibniz algebra, whose nilradical is isomorphic to $\lambda_4(\alpha).$
Then, $\ll$ is isomorphic to one of the following pairwise non-isomorphic algebras:
$$\small \begin{array}{ccccc}
 \ll_{4}(\gamma): &  \ll_{5}: & \ll_{6}:& \ll_{7}: &  \ll_{8}(\beta): \\
 \begin{cases}
 [e_2,e_1]=e_3, \\
 [e_1,x]=e_1, \\
 [e_2,x]=\gamma e_2, \ \gamma \in \mathbb{C},\\
 [e_3,x]=(1+\gamma)e_3,\\
 [x,e_1]=-e_1.
 \end{cases} &
\begin{cases}
 [e_2,e_1]=e_3, \\
 [e_1,x]=e_1, \\
 [e_2,x]=-e_2,\\
 [x,e_1]=-e_1,\\
 [x,x]=e_3.
 \end{cases}
  &
 \begin{cases}
 [e_2,e_1]=e_3, \\
 [e_1,x]=e_1+e_3, \\
 [e_3,x]=e_3,\\
 [x,e_1]=-e_1,\\
 [x,x]=-e_2.
 \end{cases} &
 \begin{cases}
 [e_2,e_1]=e_3, \\
 [e_2,x]=e_2, \\
 [e_3,x]=e_3.
 \end{cases} &
\begin{cases}
 [e_2,e_1]=e_3, \\
 [e_1,e_2]=\beta e_3,\\
 [e_1,x]=e_1, \\
 [e_2,x]=\beta e_2,\\
 [e_3,x]=(\beta +1)e_3,\\
 [x,e_1]=-e_1,\\
 [x,e_2]=-\beta e_2,
  \end{cases}
 \end{array}$$
 where $\beta=
\frac{\sqrt{1-4\alpha}-1}{\sqrt{1-4\alpha}+1},$ with $\alpha \notin \{0,\frac 1 4\}$
 \end{prop}
\begin{proof}

We have divided the proof into three steps: the first one is
relative to the study of $\lambda_4'(0),$ the other one to the
study of $\lambda_4'(\beta)$ with $\beta \neq 0,$ and the last one
to the study of $\lambda_4'.$

Let $N \hbox{ be isomrphic to
}\lambda_4'(0),$ by considering the same tools applied in Proposition \ref{propo1}, we obtain the algebras $\ll_{4}(\gamma)-\ll_{7}.$

Investigating the case $N \hbox{ isomorphic to }\lambda_4'(0),$ analogously to the previous proposition, we
have that the obtained $4$-dimensional Leibniz algebras are isomorphic to the algebras of the family $\ll_{8}(\beta).$

Finally, the study of the case $N$ isomorphic to
$\lambda_4',$ shows that there is no algebra in this case.
\end{proof}

 \begin{prop}
 Let $\ll$ be a $4$-dimensional solvable Leibniz algebra, whose nilradical is isomorphic to $\lambda_2.$
 Then, $\ll$ is isomorphic to one of the following pairwise non-isomorphic algebras:
 $$\small \begin{array}{llll}
  \ll_{9}(\gamma):\begin{cases}
 [e_1,e_2]=e_3,\\
 [e_1,x]=e_1, \\
 [e_2,x]=\gamma e_2, \\
 [e_3,x]=2e_3,\\
 [x,e_1]=-e_1,\\
 [x,e_2]=-\gamma e_2.
 \end{cases} &
\ll_{10}(\delta):\begin{cases}
 [e_1,e_1]=e_3,\\
 [e_1,x]=e_1, \\
 [e_2,x]=\delta e_2,\\
 [e_3,x]=2e_3\\
 [x,e_1]=-e_1.
 \end{cases} &
  \ll_{11}:\begin{cases}
 [e_1,e_1]=e_3,\\
 [e_1,x]=e_1, \\
 [e_3,x]=2e_3,\\
 [x,e_1]=-e_1,\\
 [x,x]=e_2.
 \end{cases} &
 \ll_{12}: \begin{cases}
 [e_1,e_1]=e_3,\\
 [e_1,x]=e_1, \\
 [e_2,x]=2e_2+e_3,\\
 [e_3,x]=2e_3,\\
 [x,e_1]=-e_1.
 \end{cases}\\ [10 mm]
 \ll_{13}:\begin{cases}
 [e_1,e_1]=e_3,\\
 [e_1,x]=e_1+e_2, \\
 [e_2,x]=e_2,\\
 [e_3,x]=2e_3,\\
 [x,e_1]=-e_1-e_2,\\
 [x,e_2]=-e_2.
 \end{cases} &
\ll_{14}:\begin{cases}
 [e_1,e_1]=e_3,\\
 [e_2,x]=e_2, \\
 [x,e_2]=-e_2.
  \end{cases} &
 \ll_{15}(\lambda): \begin{cases}
 [e_1,e_1]=e_3,\\
 [e_2,x]=e_2, \\
 [x,e_1]=e_3,\\
 [x,e_2]=-e_2,\\
 [x,x]=\lambda e_3.
 \end{cases} &
 \ll_{16}:\begin{cases}
 [e_1,e_1]=e_3,\\
 [e_2,x]=e_2, \\
 [x,e_2]=-e_2,\\
 [x,x]=-2e_3.
 \end{cases} \\ [10 mm]
 \ll_{17}: \begin{cases}
 [e_1,e_1]=e_3,\\
 [e_2,x]=e_2.
 \end{cases} &
\ll_{18}(\mu):\begin{cases}
 [e_1,e_1]=e_3,\\
 [e_2,x]=e_2, \\
 [x,e_1]=e_3,\\
 [x,x]=\mu e_3.
  \end{cases} &
 \ll_{19}: \begin{cases}
 [e_1,e_1]=e_3,\\
 [e_2,x]=e_2, \\
 [x,e_1]=e_3,\\
 [e_1,x]=e_3.
 \end{cases} &
 \end{array}$$
 where $ \gamma, \lambda, \mu \in \mathbb{C}$ and  $\delta \in \mathbb{C}\setminus \{0\}.$
 \end{prop}
 \begin{proof} The proof is analogously to the above Propositions.
 \end{proof}

\begin{prop}
 Let $\ll$ be a $4$-dimensional solvable Leibniz algebra, whose nilradical
  is $3$-dimensional  abelian algebra. Then $\ll$ is isomorphic to one of the following pairwise non isomorphic algebras
$$ \footnotesize \begin{array}{lll}
\ll_{20}(\mu_2,\mu_3):\begin{cases}
 [e_1,x]=e_1,\\
 [e_2,x]=\mu_2e_2,\\
 [e_3,x]=\mu_3e_3,\\
 [x,e_1]=-e_1,\\
 [x,e_2]=-\mu_2e_2,\\
 [x,e_3]=-\mu_3e_3,
  \end{cases}&
\ll_{21}(\mu_2, \mu_3):\begin{cases}
 [e_1,x]=e_1,\\
 [e_2,x]=\mu_2e_2,\\
 [e_3,x]=\mu_3e_3, \ \mu_3 \neq 0,\\
 [x,e_1]=-e_1,\\
 [x,e_2]=-\mu_2e_2,\\
 \end{cases}& \ll_{22}(\mu_2, \mu_3):\begin{cases}
 [e_1,x]=e_1,\\
 [e_2,x]=\mu_2e_2,\ \mu_2 \neq 0,\\
 [e_3,x]=\mu_3e_3, \ \mu_3 \neq 0,\\
 [x,e_1]=-e_1,\\
 \end{cases}\end{array}$$
  $$\footnotesize \begin{array}{lll}
 \ll_{23}(\mu_2, \mu_3):\begin{cases}
 [e_1,x]=e_1,\\
 [e_2,x]=\mu_2e_2,\\
 [e_3,x]=\mu_3e_3,\\
 \end{cases}&
\ll_{24}(\mu_2):\begin{cases}
 [e_1,x]=e_1,\\
 [e_2,x]=\mu_2e_2,\\
 [x,e_1]=-e_1,\\
 [x,e_2]=-\mu_2e_2,\\
 [x,x]=e_3.
  \end{cases}&
\ll_{25}(\mu_2):\begin{cases}
 [e_1,x]=e_1,\\
 [e_2,x]=\mu_2e_2, \ \mu_2 \neq 0,\\
 [x,e_1]=-e_1,\\
 [x,x]=e_3.
  \end{cases}\\[1mm]
  \ll_{26}(\mu_2):\begin{cases}
 [e_1,x]=e_1,\\
 [e_2,x]=\mu_2e_2,\\
 [x,x]=e_3.
  \end{cases}&\ll_{27}:\begin{cases}
 [e_1,x]=e_1,\\
 [x,e_1]=-e_1,\\
 [x,e_2]=e_3,\\
 \end{cases}&
 \ll_{28}:\begin{cases}
 [e_1,x]=e_1,\\
 [x,e_2]=e_3.\\
 \end{cases}\\[1mm]
 \ll_{29}(\mu_3):\begin{cases}
          [e_1,x]=e_1+e_2,\\
          [e_2,x]=e_2,\\
          [e_3,x]=\mu_3e_3,\\
          [x,e_1]=-e_1-e_2,\\
          [x,e_2]=-e_2,\\
                  \end{cases}&
\ll_{30}:\begin{cases}
                    [e_1,x]=e_1+e_2,\\
                    [e_2,x]=e_2,\\
                    [x,e_1]=-e_1-e_2,\\
                    [x,e_2]=-e_2,\\
                    [x,x]=e_3,
                            \end{cases}
&\ll_{31}(\mu_3):\begin{cases}
          [e_1,x]=e_1+e_2,\\
          [e_2,x]=e_2,\\
          [e_3,x]=\mu_3e_3,\\
          [x,e_1]=-e_1-e_2,\\
          [x,e_2]=-e_2,\\
          [x,e_3]=-\mu_3e_3,
                  \end{cases}\ \mu_3 \neq 0,\\[1mm]
\ll_{32}(\mu_3):\begin{cases}
          [e_1,x]=e_1+e_2,\\
          [e_2,x]=e_2,\\
          [e_3,x]=\mu_3e_3,\\
                 \end{cases}&
\ll_{33}:\begin{cases}
                    [e_1,x]=e_1+e_2,\\
                    [e_2,x]=e_2,\\
                    [x,x]=e_3,
                            \end{cases}&
\ll_{34}(\mu_3):\begin{cases}
          [e_1,x]=e_1+e_2,\\
          [e_2,x]=e_2,\\
          [e_3,x]=\mu_3e_3,\\
          [x,e_3]=-\mu_3e_3,
                  \end{cases} \ \mu_3 \neq 0, \\[1mm]
\ll_{35}(\alpha): \begin{cases}
   [e_1,x]=e_2,\\
   [e_3,x]=e_3,\\
   [x,e_1]=\alpha e_2,\\
   [x,e_3]=-e_3,
   \end{cases}&
\ll_{36}:\begin{cases}
        [e_1,x]=e_2,\\
        [e_3,x]=e_3,\\
        [x,e_1]=-e_2,\\
        [x,e_3]=-e_3,\\
        [x,x]=e_2.
        \end{cases}&
\ll_{37}:\begin{cases}
  [e_1,x]=e_2,\\
  [e_3,x]=e_3,\\
  [x,e_3]=-e_3,\\
  [x,x]=e_1.
  \end{cases}\\[1mm]
\ll_{38}(\alpha):\begin{cases}
  [e_1,x]=e_2,\\
  [e_3,x]=e_3,\\
  [x,e_1]=\alpha e_2,\\
  \end{cases}&
\ll_{39}:\begin{cases}
   [e_1,x]=e_2,\\
   [e_3,x]=e_3,\\
   [x,e_1]=-e_2,\\
   [x,x]=e_2.
   \end{cases}&
\ll_{40}:\begin{cases}
   [e_1,x]=e_2,\\
   [e_3,x]=e_3,\\
   [x,x]=e_1.
   \end{cases} \\[1mm]
   \ll_{41}:\begin{cases}
  [e_1,x]=e_1+e_2,\\
  [e_2,x]=e_2+e_3,\\
  [e_3,x]=e_3.
  \end{cases}&
\ll_{42}:\begin{cases}
  [e_1,x]=e_1+e_2,\\
  [e_2,x]=e_2+e_3,\\
  [e_3,x]=e_3,\\
  [x,e_1]=e_1-e_2,\\
  [x,e_2]=e_2-e_3,\\
  [x,e_3]=e_3,
  \end{cases}&
\end{array}$$
where, if one does not specify, $\mu_i \in \mathbb{C}.$
 \end{prop}
 \begin{proof}
 Let us consider the basis $\{x,e_1,e_2,e_3\}$ of $\ll.$ It is easy to prove that, by using simple change of basis, the matrix of $R_x$
has one of the following forms:
 $$\begin{array}{ccc}
 \left(\begin{array}{ccc}
 \mu_1 & 0 & 0\\
 0 & \mu_2 & 0\\
 0 & 0 & \mu_3
 \end{array}\right), &
  \left(\begin{array}{ccc}
\mu_1 & 1 & 0\\
0 & \mu_1 & 0\\
0 & 0 & \mu_3
 \end{array}\right), &
  \left(\begin{array}{ccc}
\mu_1 & 1 & 0\\
0 & \mu_1 & 1\\
0 & 0 & \mu_1
 \end{array}\right).
 \end{array}$$

 \

Let $R_x \simeq \left(\begin{array}{ccc}
 \mu_1 & 0 & 0\\
 0 & \mu_2 & 0\\
 0 & 0 & \mu_3
 \end{array}\right),$ then law of $\ll$ can be written as follows:

 $$\begin{cases}
 [e_1,x]=\mu_1e_1,&[x,e_1]=\alpha_1e_1+\alpha_2e_2+\alpha_3e_3,\\
 [e_2,x]=\mu_2e_2,&[x,e_2]=\beta_1e_1+\beta_2e_2+\beta_3e_3,\\
 [e_3,x]=\mu_3e_3,&[x,e_3]=\gamma_1e_1+\gamma_2e_2+\gamma_3e_3,\\
 [x,x]=\delta_1e_1+\delta_2e_2+\delta_3e_3.
 \end{cases}$$

 By the Leibniz identities $[x,[e_1,x]],$  $[x,[e_2,x]]$ and $[x,[e_3,x]],$ we obtain the following restrictions:
  \begin{equation} \label{1.1}
  \begin{array}{cc}
 \alpha_2(\mu_1-\mu_2)=0,&
 \alpha_3(\mu_1-\mu_3)=0,\\
 \beta_1(\mu_1-\mu_2)=0,&
 \beta_3(\mu_2-\mu_3)=0,\\
 \gamma_1(\mu_1-\mu_3)=0,&
 \gamma_2(\mu_2-\mu_3)=0.
 \end{array}
 \end{equation}

 By considering the Leibniz identities $[x,[x,e_1]],$ $[x,[x,e_2]]$ and $[x,[x,e_3]],$ we conclude:
\begin{equation}\label{1.2}
\begin{array}{ccc}
(\mu_1+\alpha_1)\alpha_1+\alpha_2 \beta_1+\alpha_3 \gamma_1=0,&
(\mu_1+\alpha_1)\alpha_2+\alpha_2 \beta_2+\alpha_3 \gamma_2=0,&
(\mu_1+\alpha_1)\alpha_3+\alpha_2 \beta_3+\alpha_3 \gamma_3=0,\\[1mm]
\beta_1\alpha_1+(\beta_2+\mu_2) \beta_1+\beta_3 \gamma_1=0,&
\beta_1\alpha_2+(\beta_2+\mu_2) \beta_2+\beta_3 \gamma_2=0,&
\beta_1\alpha_3+(\beta_2+\mu_2) \beta_3+\beta_3 \gamma_3=0,\\[1mm]
\gamma_1\alpha_1+ \gamma_2 \beta_1+(\gamma_3+\mu_3) \gamma_1=0,&
\gamma_1\alpha_2+ \gamma_2 \beta_2+(\gamma_3+\mu_3) \gamma_2=0,&
\gamma_1\alpha_3+ \gamma_2 \beta_3+(\gamma_3+\mu_3) \gamma_3=0,
\end{array}
\end{equation}
and by the Leibniz identity $[x,[x,x]]$ we get
\begin{equation} \label{1.3}
 \alpha_1\delta_1+\beta_1\delta_2+\gamma_1\delta_3=0,\quad
 \alpha_2\delta_1+\beta_2\delta_2+\gamma_2\delta_3=0,\quad
 \alpha_3\delta_1+\beta_3\delta_2+\gamma_3\delta_3=0.
 \end{equation}

In order to determine the remaining structural constants, we distinguish the following cases:

\textit{\bf{Case 1.}} Let $\mu_1\neq \mu_2,$ $\mu_1 \neq \mu_3,$
$\mu_2 \neq \mu_3,$ thanks to the restrictions (\ref{1.1}) and
(\ref{1.2}) we have
$$\alpha_2=\alpha_3=\beta_1=\beta_3=\gamma_1 = \gamma_2=0,$$
\begin{equation} \label{1.4}
 \begin{array}{ccc}(\mu_1+\alpha_1)\alpha_1=0, & (\beta_2+\mu_2) \beta_2=0, & (\gamma_3+\mu_3) \gamma_3=0,\\[1mm]
 \alpha_1\delta_1=0,&\beta_2\delta_2=0,& \gamma_3\delta_3=0.
 \end{array}
 \end{equation}

\textit{\bf{Case 1.1.}} Let $\mu_1\mu_2\mu_3 \neq 0,$ then we can assume $\mu_1=1,$ without loss of generality. Therefore it is worthwhile to consider the following possibilities:

\begin{itemize}
\item If $\alpha_1 \neq 0,$ $\beta_2 \neq 0,$ $\gamma_3 \neq 0,$
 according to the restriction (\ref{1.4}), we have
 $$\alpha_1 = -1, \quad \beta_2 = -\mu_2, \quad \gamma_3 = -\mu_3, \qquad \delta_1= \delta_2= \delta_3= 0.$$
Hence we obtain the algebra $\ll_{20}(\mu_2,\mu_3)$ for $\mu_2\neq 1, \
\mu_3\neq 1, \ \mu_2\neq \mu_3$ and $\mu_2\mu_3 \neq 0.$

\item If one of the parameters $\alpha_1, \beta_2, \gamma_3$ equals to zero
 and two of them unequal to zero, then there is no loss of generality in assuming $\alpha_1 \neq 0,$ $\beta_2 \neq 0$ and
 $\gamma_3=0.$ Due to the restriction (\ref{1.4}), we come to
 $$\alpha_1 = -1, \quad \beta_2 = -\mu_2, \qquad \delta_1= \delta_2= 0.$$

It suffices to make the transformation $x'=x-\frac{\delta_3}{\mu_3}e_3,$ to we get
$\delta_3=0$ and to obtain the algebra $\ll_{21}(\mu_2,\mu_3)$ for $\mu_2\neq 1, \
\mu_3\neq 1, \ \mu_2\neq \mu_3$ and $\mu_2\mu_3 \neq 0.$

\item
If two of the parameters $\alpha_1, \beta_2, \gamma_3$ equal zero
 and one of them unequals zero, then we can suppose $\alpha_1 \neq 0,$$\beta_2 = 0$ and
 $\gamma_3=0,$ without loss of generality.
By the restriction (\ref{1.4}), we have
 $$\alpha_1 = -1, \qquad \delta_1= 0.$$

By taking the transformation $x'=x-\frac{\delta_2}{\mu_2}e_2-\frac{\delta_3}{\mu_3}e_3,$
 we get $\delta_2=\delta_3=0$ and come to the algebra
 $\ll_{22}(\mu_2,\mu_3)$ for $\mu_2\neq 1, \ \mu_3\neq 1, \ \mu_2\neq
\mu_3$ and $\mu_2\mu_3 \neq 0.$

\item
If $\alpha_1 = 0,$  $\beta_2 = 0$ and
 $\gamma_3=0,$
then, thanks to the linear transformation $x'=x-\frac{\delta_1}{\mu_2}e_1-\frac{\delta_2}{\mu_2}e_2-\frac{\delta_3}{\mu_3}e_3$
 we get $\delta_1=\delta_2=\delta_3=0.$ Hence the algebra
 $\ll_{23}(\mu_2,\mu_3)$ is obtained for $\mu_2\neq 1, \ \mu_3\neq 1, \ \mu_2\neq
\mu_3$ and $\mu_2\mu_3 \neq 0$.
\end{itemize}

\textit{\bf{Case 1.2.}} Let one of the parameters $\mu_1, \mu_2, \mu_3$ equal zero then, without loss of generality,
we can assume $\mu_1=1,$ $\mu_3=0.$ Moreover, thanks to the restriction (\ref{1.4}), we have $\gamma_3=0.$

\begin{itemize}
\item Let  $\alpha_1 \neq 0,$ $\beta_2 \neq 0,$
 then due to the restriction (\ref{1.4}), we have
 $$\alpha_1 = -1, \quad \beta_2 = -\mu_2, \qquad \delta_1= \delta_2= 0.$$ It is enough to consider the following cases:
    \begin{itemize}
    \item If $\delta_3=0,$ the algebra $\ll_{20}(\mu_2,\mu_3)$ is obtained for $\mu_2
        \notin \{0,1\}$ and $\mu_3=0.$
    \item If $\delta_3\neq 0,$ by taking the basis
        transformation $e_3'=\delta_3e_3,$ we get the algebra
        $\ll_{24}(\mu_2)$ for $\mu_2 \notin \{0, 1\}.$
    \end{itemize}

\item  Let only one of the parameters $\alpha_1, \beta_2$ equals zero. Then, without loss of generality, we can assume
 $\alpha_1 \neq 0$ and $\beta_2 = 0.$ Moreover, by the restriction (\ref{1.4}) we assert
 $$\alpha_1 = -1,  \quad \delta_1= 0.$$

By making the basis transformation $x'=x-\frac{\delta_2}{\mu_2}e_2$
 we get $\delta_2=0.$ This case is completed by considering the following restrictions:
    \begin{itemize}
        \item If $\delta_3=0,$ we arrive at the algebra $\ll_{22}(\mu_2, \mu_3)$ for $\mu_2 \notin \{0, 1\}$ and $ \mu_3=0.$
        \item If $\delta_3\neq 0,$ by taking the basis
        transformation $e_3'=\delta_3e_3,$ we get to the algebra
        $\ll_{25}(\mu_2)$ for $\mu_2\notin \{0, 1\}.$
    \end{itemize}

\item  Let $\alpha_1=0, \beta_2=0.$ By taking the basis transformation $x'=x-\frac{\delta_1}{\mu_1}e_2-\frac{\delta_2}{\mu_2}e_2,$
 we obtain $\delta_1=\delta_2=0.$ It is enough to consider the following cases:
    \begin{itemize}
        \item If $\delta_3=0,$ then we get the algebra $\ll_{23}(\mu_2, \mu_3)$ for $\mu_2\notin \{0, 1\}$ and $\mu_3=0.$
        \item If $\delta_3\neq 0,$ then taking the basis
        transformation $e_3'=\delta_3e_3,$ we get the algebra
        $\ll_{26}(\mu_2)$ for $\mu_2\notin \{0, 1\}.$
    \end{itemize}
\end{itemize}

 \textit{\bf{Case 2.}} Let two of the parameters $\mu_1, \mu_2, \mu_3$ be equal. Then, without loss of generality, we can assume $\mu_1=\mu_2.$ It is interesting to distinguish the following cases:

 \textit{\bf{Case 2.1.}} Let $\mu_1=\mu_2\neq 0.$ Then it is clear that we can assume
  $\mu_1=\mu_2=1$ and $\mu_3 \neq 1.$
Moreover, due to (\ref{1.1}), (\ref{1.2}) and (\ref{1.3}), we come to the following restrictions:
$$\alpha_3=\beta_3= \gamma_1=\gamma_2=0,$$
\begin{equation}\label{1.5}
\begin{array}{lll}
(1+\alpha_1)\alpha_1+\alpha_2 \beta_1=0,&
(1+\alpha_1)\alpha_2+\alpha_2 \beta_2=0,&\\[1mm]
\beta_1\alpha_1+(1+\beta_2) \beta_1=0,&
\beta_1\alpha_2+(1+\beta_2) \beta_2=0,&
(\gamma_3+\mu_3) \gamma_3=0,\\[1mm]
 \alpha_1\delta_1+\beta_1\delta_2=0,&
 \alpha_2\delta_1+\beta_2\delta_2=0,&
\gamma_3\delta_3=0.
\end{array}
 \end{equation}

Note that for any element $y  \in \{e_1, e_2\},$ we have $[y,x] =y.$ Thus, by using any change of basis of $\{e_1, e_2\}$, it is sure that the products $ [e_1,x]=e_1$ and $[e_2,x]=e_2$ stay unchanged.

By means of appropriate basis transformation, the Jordan block of the matrix
$\left(\begin{array}{cc}\alpha_1&\alpha_2\\
\beta_1&\beta_2
\end{array}\right)$ can take either this form $\left(\begin{array}{cc}\alpha_1&0\\
0&\beta_2
\end{array}\right)$ or this one $\left(\begin{array}{cc}\alpha_1&1\\
0&\alpha_1
\end{array}\right).$

 We can consider the following possibilities:

\textit{\bf{Case 2.1.1.}} If $\left(\begin{array}{cc}\alpha_1&\alpha_2\\
\beta_1&\beta_2
\end{array}\right) \simeq \left(\begin{array}{cc}\alpha_1&0\\
0&\beta_2
\end{array}\right),$ i.e., $\alpha_2=\beta_1=0,$ then thanks to (\ref{1.5}), we have the following restrictions:
$$
(1+\alpha_1)\alpha_1, \quad (1+\beta_2) \beta_2=0, \quad (\gamma_3+\mu_3) \gamma_3=0,
$$
$$
 \alpha_1\delta_1=0,\quad
\beta_2\delta_2=0,\quad
\gamma_3\delta_3=0.
$$

It is clear that this case is similar to Case 1, with $\mu_1=\mu_2=1.$ Hence, by using similar tools, we get the following algebras:

\begin{itemize}
  \item If $\mu_3 \neq 0,$
    \begin{itemize}
              \item the algebra $\ll_{20}(\mu_2, \mu_3)$ is obtained for $\mu_2=1$ and $\mu_3\neq 1,$ by considering $\alpha_1 \neq 0,$ $\beta_2 \neq 0$ and $\gamma_3 \neq 0.$
              \item the algebra $\ll_{21}(\mu_2, \mu_3)$ is obtained for $\mu_2=1$ and $ \mu_3\neq 1,$ by considering $\alpha_1 \neq 0,$ $\beta_2 \neq 0$ and $\gamma_3 = 0.$
              \item the algebra $\ll_{21}(\mu_2, \mu_3)$ is obtained for $\mu_3=1$ and $ \mu_2\neq 1,$ by considering  $\alpha_1 \neq 0,$ $\beta_2 = 0$ and $\gamma_3 \neq 0.$
              \item the algebra $\ll_{22}(\mu_2, \mu_3)$ is obtained for $\mu_2=1$ and $\mu_3\neq 1,$ by considering $\alpha_1 \neq 0,$ $\beta_2 = 0$ and $\gamma_3 = 0.$
              \item the algebra $\ll_{23}(\mu_2, \mu_3)$ is obtained for $\mu_2=1$ and $ \mu_3\neq 1,$ by considering $\alpha_1 = 0,$ $\beta_2 = 0$ and $\gamma_3 = 0.$
    \end{itemize}

  \item If $\mu_3 = 0,$
    \begin{itemize}
             \item the algebra $\ll_{24}(\mu_2)$ is obtained for $\mu_2=1$ and $ \mu_3\neq 1,$ by considering $\alpha_1 \neq 0$ and $\beta_2 \neq 0.$
             \item the algebra $\ll_{25}(\mu_2)$ is obtained for $\mu_2=1$ and $ \mu_3\neq 1,$ by considering $\alpha_1 \neq 0$ and $\beta_2 = 0.$
             \item the algebra $\ll_{26}(\mu_2)$ is obtained for $\mu_2=1$ and $ \mu_3\neq 1,$ by considering $\alpha_1 = 0$ and $\beta_2 = 0.$
    \end{itemize}
\end{itemize}

\textit{\bf{Case 2.1.2.}} If $\left(\begin{array}{cc}\alpha_1&\alpha_2\\
\beta_1&\beta_2
\end{array}\right) \simeq \left(\begin{array}{cc}\alpha_1&1\\
0&\alpha_1
\end{array}\right),$ i.e., $\alpha_2=1, \  \beta_1=0, \ \beta_2=\alpha_2.$ Due to the restriction (\ref{1.5})
we obtain the system $(1+\alpha_1)\alpha_1=0, \quad 1+2\alpha_1=0,$ which has no solution.
Therefore in this case we do not have any algebra.

\textit{\bf{Case 2.2.}} Let $\mu_1=\mu_2=0,$ then we can take $\mu_3= 1,$ without loss of generality. It suffices to make some basis transformations to prove that $\mu_1= 1,$ $\mu_2=\mu_3=0.$

By the restrictions (\ref{1.1},)
(\ref{1.2}) and (\ref{1.3}), we have
$\alpha_2=\alpha_3=\beta_1= \gamma_1=0$ and
\begin{equation}\label{1.6}
\begin{array}{lll}
(1+\alpha_1)\alpha_1=0,&
\beta_2^2+\beta_3\gamma_2=0,&
\beta_2\beta_3+\beta_3\gamma_3=0,\\[1mm]
& \gamma_2\beta_2+\gamma_3\gamma_2=0,&
\gamma_2\beta_3+\gamma_3^2=0,\\[1mm]
 \alpha_1\delta_1=0,&
 \beta_2\delta_2+\gamma_2\delta_3=0,&
 \beta_3\delta_2+\gamma_3\delta_3=0.
 \end{array}
 \end{equation}

Analogously to Case 2.1, we consider two subcases:

\textit{\bf{Case 2.2.1.}} If $\left(\begin{array}{cc}
\beta_2&\beta_3\\[1mm]
\gamma_2& \gamma_3
\end{array}\right) \simeq \left(\begin{array}{cc}\beta_2&0\\
0&\gamma_3
\end{array}\right),$ i.e., $\beta_3=\gamma_2=0.$ Thanks to (\ref{1.6}),
we obtain the following restriction $\beta_2=\gamma_3=0,$ $(1+\alpha_1)\alpha_1=0$ and $\alpha_1\delta_1=0.$

\begin{itemize}
\item If $\alpha_1=-1,$ hence $\delta_1=0.$

If $(\delta_2,\delta_3) = (0,0),$ the algebra
$\ll_{20}(\mu_2, \mu_3)$ for $\mu_2=\mu_3=0$ is obtained.

If $(\delta_2,\delta_3) \neq (0,0),$ it is enough to make the change $e_3'=\delta_2e_2+\delta_3e_3,$
to obtain the algebra $\ll_{24}(\mu_2)$ for $\mu_2=0.$

\item If $\alpha_1=0,$ then by taking the change $x'=x-\delta_1e_1,$ we get $[x,x]=\delta_2e_2+\delta_3e_3.$

If $(\delta_2,\delta_3) = (0,0),$ the algebra $\ll_{23}(\mu_2, \mu_3)$ is obtained for $\mu_2=\mu_3=0.$

If $(\delta_2,\delta_3) \neq (0,0),$ it is suffices to make the basis transformation $e_3'=\delta_2e_2+\delta_3e_3,$ to come to the algebra $\ll_{26}(\mu_2)$ for
$\mu_2=0.$
\end{itemize}

\textit{\bf{Case 2.2.2.}} If $\left(\begin{array}{cc}
\beta_2&\beta_3\\[1mm]
\gamma_2& \gamma_3
\end{array}\right) \simeq \left(\begin{array}{cc}\beta_2&1\\
0&\beta_2
\end{array}\right),$ i.e., $\beta_3=1, \ \gamma_2=0$ and $ \gamma_3 =\beta_2.$ Due to (\ref{1.6}),
we have the following restrictions $\beta_2=0,$ $(1+\alpha_1)\alpha_1=0,$ $\alpha_1\delta_1=0$ and $\delta_2=0.$
\begin{itemize}
\item If $\alpha_1=-1,$ hence $\delta_1=0$ and the structural constant $\delta_3$ is determined by taking the change $x'=x-\delta_3e_2.$ Thus, the algebra $\ll_{27}$ is obtained.

\item If $\alpha_1=0,$ then making the change of basis $x'=x-\delta_1e_1-\delta_3e_2,$ the algebra
$\ll_{28}$ is obtained.
\end{itemize}

\textit{\bf{Case 3.}} Let $\mu_1=\mu_2=\mu_3=1,$ then we consider the following subcases.

\textit{\bf{Case 3.1.}} Let $\left(\begin{array}{ccc}\alpha_1&\alpha_2&\alpha_3\\
\beta_1&\beta_2&\beta_3\\
\gamma_1&\gamma_2&\gamma_3
\end{array}\right) \simeq \left(\begin{array}{ccc}\alpha_1&0&0\\
0&\beta_2&0\\
0&0&\gamma_3\\
\end{array}\right),$ i.e., $\alpha_2=\alpha_3=\beta_1=\beta_3=\gamma_1=\gamma_2=0.$

Thanks to the restrictions (\ref{1.2}) and (\ref{1.3}), we have
$(1+\alpha_1)\alpha_1=0,$ $(1+\beta_2) \beta_2=0,$ $(1+\gamma_3) \gamma_3=0,$ $\alpha_1\delta_1=0,$ $\beta_2\delta_2=0$ and $\gamma_3\delta_3=0.$
It is clear that this case is similar to Case 1, with $\mu_1=\mu_2=\mu_3=1.$ Hence, by using similar tools, we get the following algebras:

 \begin{itemize}
   \item $\ll_{20}(\mu_2, \mu_3)$ for $\mu_2=\mu_3=1,$ by considering $\alpha_1 \neq 0,$ $\beta_2 \neq 0$ and $\gamma_3 \neq 0.$
   \item $\ll_{21}(\mu_2, \mu_3)$ for $\mu_2=\mu_3=1,$ by considering $\alpha_1 \neq 0,$ $\beta_2 \neq 0$ and $\gamma_3 = 0.$
   \item $\ll_{22}(\mu_2, \mu_3)$ for $\mu_2=\mu_3=1,$ by considering $\alpha_1 \neq 0,$ $\beta_2 = 0$ and $\gamma_3 = 0.$
   \item $\ll_{23}(\mu_2, \mu_3)$ for $\mu_2=\mu_3=1,$ by considering $\alpha_1 = 0,$ $\beta_2 = 0$ and $\gamma_3 = 0.$
 \end{itemize}

 \textit{\bf{Case 3.2.}} Let $\left(\begin{array}{ccc}\alpha_1&\alpha_2&\alpha_3\\
\beta_1&\beta_2&\beta_3\\
\gamma_1&\gamma_2&\gamma_3
\end{array}\right) \simeq \left(\begin{array}{ccc}\alpha_1&1&0\\
0&\alpha_1&0\\
0&0&\gamma_3\\
\end{array}\right),$ i.e., $\alpha_2=1,$ $ \alpha_3=\beta_1=\beta_3=\gamma_1=\gamma_2=0,$ and $\beta_2=
\alpha_1.$ By the restrictions (\ref{1.2}) and (\ref{1.3}), we obtain the following system:
$$(1+\alpha_1)\alpha_1=0, \quad 1+2\alpha_1=0,
$$
which has no solution. Therefore in this case we do not obtain any algebra.

 \textit{\bf{Case 3.3.}} Let $\left(\begin{array}{ccc}\alpha_1&\alpha_2&\alpha_3\\
\beta_1&\beta_2&\beta_3\\
\gamma_1&\gamma_2&\gamma_3
\end{array}\right) \simeq \left(\begin{array}{ccc}\alpha_1&1&0\\
0&\alpha_1&1\\
0&0&\alpha_1\\
\end{array}\right),$ i.e., $\alpha_2=\beta_3=1, \ \alpha_3=\beta_1=\gamma_1=\gamma_2=0, \ \beta_2=\gamma_3=
\alpha_1.$ Due to the restrictions (\ref{1.2}) and (\ref{1.3}), we have the following system:
$$(1+\alpha_1)\alpha_1=0, \quad 1+2\alpha_1=0,$$ which has no solution.

 \

Similarly, if $R_x \simeq \left(\begin{array}{ccc}
 \mu_1 & 1 & 0\\
 0 & \mu_1 & 0\\
 0 & 0 & \mu_3
 \end{array}\right),$ we obtain the algebras $\ll_{29}(\mu_3)-\ll_{40}.$ On the other hand, if
  $R_x\simeq \left(\begin{array}{ccc}
 \mu_1 & 1 & 0\\
 0 & \mu_1 & 1\\
 0 & 0 & \mu_1
 \end{array}\right),$
  the algebras $\ll_{41}, \ll_{42}$ are obtained.
 \end{proof}

  \begin{prop}
 Let $\ll$ be a $4$-dimensional solvable Leibniz algebra, whose nilradical is isomorphic to $\lambda_6.$
  Then $\ll$ is isomorphic to the following algebra
 $$\ll_{43}:
 \begin{cases}
 [e_1,e_1]=e_2,&[e_1,x]=e_1,\\
 [e_2,e_1]=e_3,& [e_2,x]=2e_1,\\
 [x,e_1]=-e_1,&[e_3,x]=3e_1,
 \end{cases}$$
 \end{prop}
\begin{proof}
The algebra $\ll_{43}$ is obtained by using similar tools to the above propositions.
\end{proof}

 \begin{prop}
 Let $\ll$ be a $4$-dimensional solvable Leibniz algebra, whose nilradical is isomorphic to $\lambda_5.$
 Then $\ll$ is isomorphic to the following algebra
 $$\ll_{44}:
 \begin{cases}
 [e_1,e_2]=e_3,& [e_1,x]=e_1,\\
 [e_2,e_1]=e_3,& [e_2,x]=e_2,\\
 [x,e_1]=-e_1,&  [e_3,x]=2e_3,\\
 [x,e_2]=-e_2.
 \end{cases}$$
 \end{prop}
\begin{proof}
The algebra $\ll_{44}$ is obtained by using similar tools to the above propositions.
\end{proof}

\begin{rem} We have used a computer program, implemented in the software \textit{Mathematica}, which allows us to check that the obtained algebras in this work are no isomorphic. The algorithmic method of this program can be found in detail in \cite{Computing}.
\end{rem}

%
%
%
%
%
%
%
%
%
%

\

{\sc Elisa M. Ca\~{n}ete.}  Dpto. Matem\'{a}tica Aplicada I.
Universidad de Sevilla. Avda. Reina Mercedes, s/n. 41012 Sevilla.
(Spain), e-mail: \emph{elisacamol@us.es}

\

{\sc Abror Kh. Khudoyberdiyev.} Institute of Mathematics, Do'rmon yo'li str. 29, 100125, Tashkent
(Uzbekistan), e-mail: \emph{khabror@mail.ru}

\end{document}